\documentclass[12pt]{amsart}

\usepackage[T1]{fontenc}
\usepackage{amsmath,amssymb}
\usepackage{amscd}
\usepackage{pb-diagram}
\newtheorem{prop}{Theorem}[section]
\newtheorem{lemma}[prop]{Lemma}

\newtheorem{definition}[prop]{Definition}

\newtheorem{remark}[prop]{Remark}
\newcommand{\clc}{\cdot\ldots\cdot}

\newcommand{\olo}{\otimes\ldots\otimes}
\newcommand{\plp}{+ \ldots +}

\newcommand{\nn}{\mathbb{N}}
\newcommand{\zz}{\mathbb{Z}}
\newcommand{\Qq}{\mathbb{Q}}

\newcommand{\cc}{\mathbb{C}}
\newcommand{\C}[1]{\mathcal{#1}}

\newcommand{\T}[1]{\textrm{#1}}
\newcommand{\E}[1]{\emph{#1}}

\newcommand{\F}[1]{\mathbf{#1}}
\newcommand{\fork}[2]{\left\{ \begin{array}{#1} #2 \end{array} \right.} 

\newcommand{\arr}[2]{\begin{array}{#1} #2 \end{array}}
\newcommand{\mat}[2]{\left(\begin{array}{#1} #2 \end{array} \right)}
\newcommand{\q}{\qquad}
\newcommand{\qq}{\qquad \qquad}

\newcommand{\qqqq}{\qquad \qquad \qquad \qquad}

\newcommand{\grad}{\hat\otimes}

\newcommand{\wit}[1]{\widetilde{#1}}
\newcommand{\wih}[1]{\widehat{#1}}

\numberwithin{equation}{section}

\begin{document}
\title{Invariance results for pairings with algebraic $K$-theory}

\author{J. Kaad}
\maketitle
\vspace{-10pt}
\centerline{ Department of Mathematics, Australian National University}
\centerline{ John Dedman Building, Acton, ACT, $0200$ Australia}
\centerline{ email: jenskaad@hotmail.com}

\bigskip
\vspace{30pt} 

\centerline{\textbf{Abstract}} 
To each algebra over the complex numbers we associate a sequence of abelian
groups in a contravariant functorial way. In degree $(m-1)$ we have the
$m$-summable Fredholm modules over the algebra modulo stable $m$-summable
perturbations. These new finitely summable $K$-homology groups pair with
cyclic homology and with algebraic $K$-theory. In the case of cyclic homology
the pairing is induced by the Chern-Connes character. The pairing between
algebraic $K$-theory and finitely summable $K$-homology is induced by the
Connes-Karoubi multiplicative character.

\newpage
\tableofcontents
\newpage

\section{Introduction}
The aim of this paper is to study the invariance properties of the
Connes-Karoubi multiplicative pairing between finitely summable Fredholm
modules and algebraic $K$-theory,
\begin{equation}\label{eq:multpairin}
\arr{ccc}{
\C M : \T{Ell}^{m-1}(A) \times K_m(A) 
\to \cc/(2 \pi i)^{\lceil m/2 \rceil} \zz & \q & m \in \nn
}
\end{equation}
Here $\T{Ell}^{m-1}(A)$ denotes the abelian group of unitary equivalence
classes of $m$-summable Fredholm modules over a $\cc$-algebra $A$ together
with a natural choice of inverses. The abelian group $K_m(A)$ is the algebraic
$K$-theory of $A$. The multiplicative pairing was constructed by A. Connes and
M. Karoubi in \cite{conkar}. The associated invariants of algebraic $K$-theory
\[
\arr{ccc}{
\C M_{\C F} : K_m(A) \to \cc/(2 \pi i)^{\lceil m/2 \rceil} \zz & \q & \C F \in
\T{Ell}^{m-1}(A)
}
\]
can be interpretted as higher determinants induced by $m$-summable Fredholm
modules. Indeed, when $m = 1$, the multiplicative character is constructed
from the application
\[
\arr{ccc}{
(g,h) \mapsto \T{det}(gh^{-1}) & \q & g,h \in \C L(H)^*, \, \, g-h \in \C L^1(H)
}
\]
where $\T{det} : GL_1(\C L^1(H)) \to \cc^*$ denotes the Fredholm
determinant. See \cite[\S $5$]{conkar}. When $m=2$, the multiplicative
character can be described in terms of the determinant invariant of
L. Brown. This invariant is essentially given by the homomorphism
\[
\delta := (\T{det} \circ \partial) : K_2(\C L(H)/\C L^1(H)) \to K_1(\C
L^1(H),\C L(H)) \to \cc^*
\]
Here the boundary map comes from the short exact sequence of $\cc$-algebras
\[
\begin{CD}
0 @>>> \C L^1(H) @>>> \C L(H) @>>> \C L(H)/\C L^1(H) @>>> 0
\end{CD}
\]
and the homomorphism $\T{det} : K_1(\C L^1(H),\C L(H)) \to \cc^*$ is the
Fredholm determinant on relative algebraic $K$-theory. See \cite{brownII}. The
precise relation between the multiplicative character and the determinant
invariant can be found in \cite{kaadI}, see also \cite[\S $5$]{conkar}. The
determinant invariant has found an interesting application in the work of
R. Carey and J. Pincus on generalizations of the Szegö limit theorem to the
case of non-zero winding numbers, see \cite{cpI,cpII}. When $m \geq 3$ and $A$
is a commutative Banach algebra, the multiplicative character can be
calculated on higher Loday products of exponentials. The result is expressed
in terms of the index cocycle applied to permutations of logarithms. The
precise formula reads
\begin{equation}\label{eq:calcformula}
\C M_{\C F}([e^{a_0}] * \ldots * [e^{a_{m-1}}])
= \sum_{s \in \Sigma_{m-1}}(q \circ \tau_{\C F})
(c_{m-1} \, \T{sgn(s)}b_0 \otimes b_{s(1)} \olo b_{s(m-1)})
\end{equation}
Here $\tau_{\C F} : C^\lambda_{m-1}(A) \to \cc$ denotes the index cocycle
of the $m$-summable Fredholm module $\C F \in \T{Ell}^{m-1}(A)$, $q : \cc \to
\cc /(2 \pi i)^{\lceil m/2 \rceil} \zz$ is the quotient map and $c_{m-1} \in
\Qq$ is a rational constant. The notation $* : K_n(A) \times K_k(A) \to
K_{n+k}(A)$ refers to the interior Loday product. The elements
$b_0,\ldots,b_{m-1} \in A$ are any logarithms of the exponentials
$e^{a_0},\ldots,e^{a_{m-1}}$,
\[
\arr{ccc}{
e^{b_i} = e^{a_i} & \T{for all} & i \in \{0,\ldots,m-1\}
}
\]
See \cite{kaadII}.

Having these computational results for the multiplicative character in mind,
it seems desirable to answer the following question:

\emph{Which reasonable equivalence relations can be put on the abelian groups
  of finitely summable Fredholm modules without changing the multiplicative
  pairing \eqref{eq:multpairin}?}

As the formula \eqref{eq:calcformula} suggests, an important part of the
multiplicative character is the index cocycle $\tau_{\C F} :
C^\lambda_{m-1}(A) \to \cc$ associated with the $m$-summable Fredholm
module. In fact, it turns out that the study of the invariance properties of
the multiplicative character essentially amounts to a study of the invariance
properties of the Chern character
\[
\arr{ccc}{
\T{Ch} : \T{Ell}^{m-1}(A) \to HC^{m-1}(A) & \q & 
\T{Ch}(\C F) = [\tau_{\C F}]
}
\]
Here $HC^{m-1}(A)$ denotes the cyclic cohomology of the $\cc$-algebra $A$ and
$[\tau_{\C F}]$ denotes the class of the index cocycle. It should be noticed
that the target for the above Chern character is one of the non-periodic
versions of cyclic cohomology. If the target is replaced by periodic cyclic
cohomology then there are good invariance results available. Here periodic
cyclic cohomology, $HP^0(A)$ and $HP^1(A)$, is defined as the inductive limit
of cyclic cohomology under the periodicity operators $S : HC^*(A) \to
HC^{*+2}(A)$. For example, it has been proved by A. Connes that the
"topological" Chern character
\[
\arr{ccc}{
\underline{\T{Ch}} : \T{Ell}^{m-1}(A) \to HP^j(A) & \q & j \T{ parity as }m-1
}
\]
is invariant under appropriate differentiable fields of $m$-summable Fredholm
modules, see \cite[I $5$ Lemma $1$]{connes} for a precise statement of this
result. A study of these homotopy invariance properties in the bivariant case
has been carried out by V. Nistor, see \cite{nistorI}. In general, the Chern
character from smooth homotopy classes of Fredholm modules to periodic cyclic
cohomology is constructed to analyze data coming from topological $K$-theory,
or more specifically for index-theoretical considerations. When dealing with
the more subtle algebraic $K$-theory it is however crucial to avoid any
application of the periodicity operator. For example, when $A$ is commutative,
the cyclic cycle given by the sum over all permutations
\[
\arr{ccc}{
x = \sum_{s \in \Sigma_{m-1}}\T{sgn}(s) b_0 \otimes b_{s(1)} \olo b_{s(m-1)}
\in Z_{m-1}^\lambda(A) & \q & b_0,\ldots,b_{m-1} \in A
}
\]
lies in the kernel of the periodicity operator in cyclic homology. But this
specific element shows up in our explicit formula \eqref{eq:calcformula} for
the multiplicative character. In order to analyze the information produced by
secondary invariants of algebraic $K$-theory it is therefore necessary to
obtain invariance results for the more refined Chern character
\[
\T{Ch} : \T{Ell}^{m-1}(A) \to HC^{m-1}(A)
\]
and not only for its topological sibling. The similarity between finitely
summable Fredholm modules and generators of analytic $K$-homology suggests
that the appropriate equivalence relation on $\T{Ell}^{m-1}(A)$ has a
$K$-homological analogue. Furthermore, we should choose the equivalence
relation with the most algebraic flavour. A good candidate could thus be a
finitely summable version of stable compact perturbation. See \cite[Definition
$17.2.4$]{black}. The next Theorem, which we prove in Section
\ref{cobordpert}, indicates that this is an appropriate choice:

\begin{prop}\label{statemaini}
Let $\C F = (\pi,H,F)$ and $\C G = (\pi,H,G)$ be two $m$-summable Fredholm
modules over a $\cc$-algebra $A$. Assume that the difference $G-F \in \C
L^m(H)$ lies in the $m^{\E{th}}$ Schatten ideal. Then the Chern characters of
$\C F$ and $\C G$ coincide in cyclic cohomology,
\[
\E{Ch}(\C F) = \E{Ch}(\C G) \in HC^{m-1}(A)
\]
\end{prop}

Let us give a few comments on the proof of this result. We apply the language
of generalized chains and generalized cycles together with the explicit
formulas for their Chern characters in the $(b,B)$-complex, due to
A. Gorokhovsky, see \cite{gorok}. For any generalized chain $\F \Omega$ with
boundary $\F{\partial \Omega}$ we have the identity
\[
(b + B)(\T{Ch}(\F \Omega)) = S(\T{Ch}(\F{\partial \Omega}))
\]
in the $(b,B)$-complex, $\C B^*(A)$. Thus, the Chern character of the boundary
of a generalized chain is a cyclic coboundary after application of the
periodicity operator. This formula only seems to give information on the
vanishing of Chern characters in periodic cyclic cohomology. However, when the
toppart of the cyclic cochain $\T{Ch}(\F \Omega) =
(\varphi^m,\varphi^{m-2},\ldots) \in \C B^m(A)$ vanishes, we actually get that
the Chern character of the boundary $\T{Ch}(\F{\partial \Omega}) \in \C
B^{m-1}(A)$ is a cyclic coboundary. This observation was used by
A. Gorokhovsky to give a proof of an invariance result similar to Theorem
\ref{statemaini} in the even case and with slightly stronger summability
conditions, see \cite[Remark $4.1$]{gorok}. One should also remark the
invariance result of A. Connes in the odd case, see \cite[I $7$ Proposition
$4$]{connes}. In order to prove Theorem \ref{statemaini} we construct a
generalized chain $\F \Omega_{\C T}$ such that
\begin{enumerate}
\item The toppart of the Chern character $\T{Ch}(\F {\Omega_{\C T}})$
  vanishes.
\item The Chern character of the boundary 
\[
\T{Ch}(\F { \partial \Omega_{\C T}}) 
= \frac{1}{(m-1)!}(\tau_{\C G} - \tau_{\C F})
\]
is the difference of the index cocycles of the Fredholm modules $\C G$ and $\C
F$ up to a constant.
\end{enumerate}

Encouraged by the achievement of Theorem \ref{statemaini} we define a new
$K$-homology type functor. As mentioned earlier, the main idea consists of
replacing the compactness condition on Kasparov modules by the condition of
finite summability. Furthermore, the equivalence relation is given by a
finitely summable version of stable compact perturbation. We also distinguish
between even and odd by means of grading operators. Thus, for any algebra $A$
over $\cc$ we define a sequence of abelian groups
\[
\arr{ccc}{
FK^{m-1}(A) := \T{Ell}^{m-1}(A) / \sim_{fc} & \q & m \in \nn
}
\]
In dimension $(m-1)$ we have the $m$-summable Fredholm modules modulo the
equivalence relation of stable $m$-summable perturbation. Because of the
resemblance with the $K$-homology of $C^*$-algebras we call this abelian group
the \emph{$m$-summable $K$-homology} of $A$. Notice that we do not consider
operator homotopic finitely summable Fredholm modules to be equal. We expect
that such an equivalence relation would be to weak for the appropriate pairing
with algebraic $K$-theory to exist. See Theorem \ref{statemainii}.

The invariance result of Theorem \ref{statemaini} entails that the
Chern-Connes character of finitely summable Fredholm modules descends to a
natural homomorphism of degree $0$,
\[
\arr{ccc}{
\T{Ch} : FK^{m-1}(A) \to HC^{m-1}(A) & \q & \T{Ch}([\C F]) =
[\tau_{\C F}]
}
\]
In particular, we get a well-defined pairing
\begin{equation}\label{eq:cychompairity}
\tau : FK^*(A) \times HC_*(A) \to \cc
\end{equation}
between finitely summable $K$-homology and cyclic homology.

Furthermore, as hinted at earlier, the well-definedness of the pairing
\eqref{eq:cychompairity} essentially entails a factorization result for the
Connes-Karoubi multiplicative pairing. This is the main result of the present
paper and the primary reason for the introduction of finitely summable
$K$-homology.

\begin{prop}\label{statemainii}
The multiplicative character associated to $m$-summable Fredholm modules over
$A$ induces a pairing
\[
\C M : FK^{m-1}(A) \times K_m(A) \to \cc/(2 \pi i)^{\lceil m/2 \rceil}\zz
\]
between the $m$-summable $K$-homology group and the $m^{\E{th}}$ algebraic
$K$-group.
\end{prop}

The finitely summable $K$-homology groups can be equipped with some
interesting homomorphisms. First of all there are the natural periodicity
homomorphisms
\[
S : FK^*(A) \to FK^{*+2}(A)
\]
which are induced by the forgetful map related to the inclusion of Schatten
ideals $\C L^*(H) \subseteq \C L^{*+2}(H)$. The periodicity homomorphisms in
finitely summable $K$-homology are paralleled by the periodicity homomorphisms
in cyclic cohomology.

Furthermore, when $A$ is a pre-$C^*$-algebra there are natural comparison
homomorphisms
\[
\arr{ccc}{
\alpha : FK^{m-1}(A) \to K^j(\, \overline A \,)
& \q & j \in \{0,1\} \T{ parity as }m-1
}
\]
from finitely summable $K$-homology to analytic $K$-homology. Here $\overline
A$ denotes the $C^*$-algebra closure of $A$. This comparison homomorphism
comes from the inclusion of the Schatten ideals into the compact operators.

It is not unlikely that the comparison homomorphism becomes an isomorphism
from a certain degree for many interesting pre-$C^*$-algebras. In this
respect, compare with the case of cyclic cohomology and periodic cyclic
cohomology of the smooth functions on a manifold.

Let us finish this introduction by making some remarks on an earlier
definition which is related to our finitely summable $K$-homology.

In the papers, \cite{wangI,wangII}, X. Wang deals with contravariant functors
from a category of smooth algebras to abelian semi groups. These functors are
denoted by $\T{Ext}_\tau$ and depend on the choice of a certain operator ideal
$\C K_\tau$. The purpose is to classify smooth extensions of smooth algebras
by these operator ideals. When $\C K_\tau = \C L^m(H)$ and $m = 2k$ is even,
the definition of the functor $\T{Ext}_\tau$ is strongly related to the
definition of the $m$-summable $K$-homology functor, $FK^{m-1}$, which we will
introduce in the sequel. This provides a link between the present paper and
the work on smooth extensions carried out by R. Douglas, D. Voiculescu,
G. Gong and others, see \cite{douglas,douglasvoicu,gong}, for example.

We have organized our material as follows.

In the first section we consider the Chern-Connes character which maps a
finitely summable Fredholm to the class of its index cocycle in cyclic
cohomology. We are then able to prove that this homomorphism is invariant
under finitely summable perturbations of the finitely summable Fredholm
module.

In the second section we start by introducing the finitely summable
$K$-homology groups. We then discuss the invariance properties of the
secondary invariants which are associated to finitely summable Fredholm
modules. This includes proving the existence of a pairing between finitely
summable $K$-homology and algebraic $K$-theory which is induced by the
Connes-Karoubi multiplicative character.

{\bf Acknowledgements:} I would like to thank Ryszard Nest for the many
inspiring conversations which we have had at the University of Copenhagen. I
am also very grateful to Uuye Otgonbayar for drawing my attention to the
article \cite{gorok} of A. Gorokhovsky. This has helped me a lot in
simplifying the arguments of the paper.





\section{Finitely summable Chern characters and perturbations}\label{invarcyc}
In this section we study the invariance properties of the Chern-Connes
character
\[
\arr{ccc}{
\T{Ch} : \T{Ell}^{m-1}(A) \to HC^{m-1}(A) & \q & \T{Ch}(\C F) = [\tau_{\C F}]
}
\]
which sends an $m$-summable Fredholm module to the class of its index cocycle
in cyclic cohomology times a constant. Our main result is that the
Chern-Connes characters of two $m$-summable Fredholm modules $\C F =
(\pi,H,F)$ and $\C G = (\pi,H,G)$ coincide whenever the difference $G-F \in \C
L^m(H)$ lies in the $m^{\T{th}}$ Schatten ideal,
\[
\T{Ch}(\C F) = [\tau_{\C F}]
= [\tau_{\C G}] = \T{Ch}(\C G) \in HC^{m-1}(A)
\]
We construct an explicit cyclic cochain with coboundary equal to the
difference of the index cocycles.

In the first three subsections we review the following material:
\begin{enumerate}
\item Basic definitions in cyclic cohomology.
\item The explicit formulas in the $(b,B)$-bicomplex for Chern characters of
  generalized chains and generalized cycles.
\item The universal differential graded algebra associated to a graded
  algebra.
\end{enumerate}

In the last subsection we prove the invariance of the Chern character of an
$m$-summable Fredholm module under $m$-summable perturbations. The proof
relies heavily on the material covered in the first three subsections.



\subsection{Cyclic cohomology}\label{cyccohom}
This section contains a short review of the construction of cyclic cohomology
by means of the $(b,B)$-bicomplex. For more details on these matters we refer
to the book of A. Connes, \cite{connesII}.

Let $A$ be a unital $\cc$-algebra. The Hochschild cochains of degree $m \in
\nn_0$, $C^m(A)$, consist of the multilinear maps of $(m+1)$ variables
\[
C^m(A) := \T{Hom}_\cc(A \otimes A^{\otimes m},\cc)
\]
The Hochschild cochains can be equipped with two differentials: The Hochschild
coboundary $b : C^m(A) \to C^{m+1}(A)$ and the Connes coboundary $B : C^m(A)
\to C^{m-1}(A)$. Both of these operators have explicit descriptions and
satisfy the relations
\[
\arr{ccccc}{
b^2 = 0 & \q & B^2 = 0 & \q & bB = - Bb
}
\]
We can thus form the bicomplex
\[
\begin{CD}
\vdots & & \vdots \\
@AbAA @AbAA \\
C^m(A) @>B>> C^{m-1}(A) @>B>> \ldots \\
@AbAA @AbAA \\
C^{m-1}(A) @>B>> C^{m-2}(A) @>B>> \ldots \\
\vdots & & \vdots \\
@AbAA @AbAA \\
C^1(A) @>B>> C^0(A) \\
@AbAA \\
C^0(A)
\end{CD}
\]
which we will refer to as the $(b,B)$-bicomplex. The totalization of the
$(b,B)$-bicomplex is denoted by $\C B^*(A)$ and the cyclic cohomology is the
cohomology of this complex,
\[
HC^*(A) = H^*\big(\C B^*(A)\big)
\]
A cyclic $m$-cochain $\varphi \in \C B^m(A)$ is thus given by a sequence of
homomorphisms 
\[
\varphi = (\varphi^m,\varphi^{m-2},\ldots)
\]
and it is a cyclic $m$-cocycle precisely when
\[
(b+B)\varphi = (b\varphi^m,B\varphi^m + b\varphi^{m-2},\ldots) = 0
\]

The cyclic cohomology comes equipped with periodicity operators,
\[
\arr{ccc}{
S : HC^m(A) \to HC^{m+2}(A) & \q & S(\varphi^m,\varphi^{m-2},\ldots) =
(0,\varphi^m,\varphi^{m-2},\ldots)
}
\]
which are induced by a shift in the $(b,B)$-bicomplex.

The definition of cyclic cohomology can be extended to a general $\cc$-algebra
$A$. Let $\wit A$ denote the unitalization of $A$. The Hochschild cochains
over $\wit A$ can be restricted to Hochschild cochains over $\cc$ by means of
the inclusion $i : \cc \to \wit A$. This operation induces a cochain
homomorphism
\[
i^* : \C B^*(\wit A) \to \C B^*(\cc)
\]
By a slight abuse of notation we let $\C B^*(A) := \T{Ker}(i^*)$ denote the
kernel complex. The cyclic cohomology of $A$ is then defined as the cohomology
of $\C B^*(A)$, $HC^*(A) := H^*(\C B^*(A))$. When $A$ has a unit, the groups
obtained by this more general definition are isomorphic to the cyclic
cohomology groups defined above.

The cyclic cohomology is a contravariant functor from the category of algebras
over $\cc$ to the category of vector spaces over $\cc$.

On some occasions in the sequel we will make use of a continuous version of
cyclic cohomology. Let $A$ be a unital locally convex topological algebra with
topology defined by a fundamental system of seminorms $\{p_i\}_{i \in
  I}$. The continuous Hochschild $m$-cochains $C^m_{\T{cont}}(A)$ consists of
the Hochschild $m$-cochains $\varphi \in C^m(A)$ which are continuous in the
sense that there exists an $i \in I$ and a constant $C \in [0,\infty)$ with
\[
\arr{ccc}{
|\varphi(a_0,\ldots,a_m)| \leq C p_i(a_0)\clc p_i(a_m) 
& \T{for all} & a_0,\ldots,a_m\in A
}
\]
The continuous cyclic cohomology is then defined precisely as the algebraic
version of cyclic cohomology except that the Hochschild cochains are replaced
by continuous Hochschild cochains. The continuous cyclic cohomology is denoted
by $HC^*_{\T{cont}}(A)$. The definition of continuous cyclic cohomology is
extended to non-unital locally convex topological algebras in the same way as
algebraic cyclic cohomology. The continuous cyclic cohomology is a
contravariant functor from locally convex topological algebras to vector
spaces over $\cc$.



\subsection{Generalized chains and cobordisms}\label{genchaincobord}
In this section we will review the notions of generalized chains and
generalized cycles as well as the explicit formulas for their Chern
characters. The concrete formulas which we present was introduced by
A. Gorokhovsky, \cite{gorok}, whereas the general scheme of ideas is due to
A. Connes, \cite{connes,connesII}. Our exposition will follow \cite{gorok}
closely, but see also \cite{nistor}.

Let $A$ be a unital algebra over $\cc$. By a \emph{generalized chain}
$\mathbf{\Omega}=(\Omega,\partial\Omega,\nabla, \partial\nabla,\int)$ of
dimension $m \in \nn_0$ over $A$ we shall understand the following data:
\begin{enumerate}
\item Two unital graded algebras $\Omega = \oplus_{n=0}^\infty \Omega_n$ and
  $\partial\Omega = \oplus_{n=0}^\infty \partial \Omega_n$ together with a
  unital surjective graded algebra homomorphism $r : \Omega \to \partial
  \Omega$ and a unital algebra homomorphism $\rho : A \to \Omega_0$.
\item Two graded derivations of degree one (termed \emph{connections}) $\nabla
  : \Omega_* \to \Omega_{*+1}$ and $\partial \nabla : \partial \Omega_* \to
  \partial \Omega_{*+1}$ together with an element $\theta \in \Omega_2$ such
  that 
\[ 
\arr{ccccc}{
r \circ \nabla = \partial \nabla \circ r & \q & \nabla^2(\omega) = \theta
\omega - \omega \theta & \q & \nabla(\theta) = 0
}
\]
The element $\theta \in \Omega_2$ is called the \emph{curvature} of the
connection $\nabla$. 
\item A graded trace $\int : \Omega_m \to \cc$ satisfying
\[
\arr{ccc}{
\int \nabla \omega = 0 & \q & 
\forall \omega \in \Omega_{m-1} \T{ with } r(\omega)= 0
}
\]
\end{enumerate}

A generalized chain with $\partial \Omega = 0$ will be called a
\emph{generalized cycle}. A generalized chain with curvature $\theta = 0$ will
be called a \emph{chain}. A generalized chain with $\partial \Omega = 0$ and
curvature $\theta = 0$ will be called a \emph{cycle}. 
 
By the \emph{boundary} of the generalized chain $\mathbf{\Omega}=(\Omega, \partial
\Omega, \nabla,\partial \nabla, \int)$ we shall understand the generalized
cycle $\mathbf{\partial \Omega} = (\partial \Omega, \partial \nabla, \int')$ over
$A$ of dimension $(m-1)$. Here the graded trace $\int' : \partial \Omega_{m-1}
\to \cc$ is given by
\[
\arr{ccc}{
\int' \sigma = \int \nabla \omega & \T{for any} & \omega \in
\Omega_{m-1} \T{ with } r(\omega) = \sigma
}
\]
The unital algebra homomorphism $\partial \rho : A \to \partial \Omega_0$ is
defined as the composition $\partial \rho = r \circ \rho$. The curvature of
the boundary is given by $\partial \theta = r(\theta)$.

\begin{definition}
We say that two generalized cycles over $A$, $\bf{\Omega}^0$ and
$\bf{\Omega}^1$, are \emph{cobordant} if there exists a generalized chain
$\bf{\Omega}$ with boundary $\bf{\Omega}^1 \oplus \underline{\bf{\Omega}}^0$
such that $r \circ \rho = (\rho_1, \rho_0)$. Here $\underline{\bf{\Omega}}^0$
is obtained from $\bf{\Omega}^0$ by changing the sign of the graded trace
$\int^0$.
\end{definition}

The relation of cobordancy is an equivalence relation between generalized
cycles, it will be denoted by $\sim_{\T{co}}$.

To each generalized chain $\mathbf{\Omega}$ of dimension $m$ we associate
a cyclic cochain 
\[
\T{Ch}(\mathbf{\Omega}) 
= \big(\T{Ch}(\mathbf{\Omega})^m, \T{Ch}(\mathbf{\Omega})^{m-2},\ldots \big)
\in \C B^m(A)
\]
The cochain $\T{Ch}(\mathbf{\Omega})$ is called the \emph{Chern character} of
the generalized chain and is given by the following JLO-type formula
\[
\begin{split}
& \T{Ch}(\mathbf{\Omega})^{m-2k}(a_0,\ldots,a_{m-2k}) \\
& \q = 
(-1)^k / (m-k)! \sum_{i_0 \plp i_{m-2k} = k}
\int \rho(a_0)\theta^{i_0}\nabla(\rho(a_1))\theta^{i_1} \ldots
\nabla(\rho(a_{m-2k}))\theta^{i_{m-2k}} \\
& \qqqq k = 1,\ldots, \lfloor m/2 \rfloor
\end{split}
\]
This explicit description is due to A. Gorokhovsky, \cite{gorok}. The next
theorem clarifies the relation between the Chern character of a generalized
chain and the Chern character of its boundary.

\begin{prop}\cite[Theorem $2.1$]{gorok}\label{cobordant}
For each generalized chain $\mathbf{\Omega}$ of dimension $m$ over the unital
$\cc$-algebra $A$ with boundary $\mathbf{\partial \Omega}$ we have the
identity
\[
(B + b)\E{Ch}(\mathbf{\Omega}) = S \E{Ch}(\mathbf{\partial \Omega})
\]
in $\C B^{m+1}(A)$. 
\end{prop}

In particular, the Chern character of a generalized cycle over $A$ is a cyclic
\emph{cocycle} in $\C B^*(A)$. Furthermore, the difference of the Chern
characters associated with cobordant generalized cycles lies in the kernel of
the periodicity operator. In fact, as we shall see in the next theorem, this
characterizes the relation of cobordancy between generalized cycles.

\begin{prop}
Any two generalized cycles $\F{\Omega^0}$ and $\F{\Omega^1}$ of dimension
$m-1$ over the unital $\cc$-algebra $A$ are cobordant if and only if the
difference of their Chern characters in cyclic cohomology lies in the kernel
of the periodicity operator. That is
\[
\F \Omega^0 \sim_{\E{co}} \F \Omega^1 \Leftrightarrow [\E{Ch}(\F \Omega^0)] -
[\E{Ch}(\F \Omega^1)] \in \E{Ker}(S)
\]
\end{prop}
\begin{proof}
This follows from \cite[Theorem $2.6$]{gorok} and \cite[III $1.\beta$ Theorem
$21$]{connesII}.
\end{proof}



\subsection{The universal differential graded algebra of a graded
  algebra}\label{udgaga}
We will now recall the construction of universal differential graded algebras
of graded algebras. This is a natural adaption to the graded case of the
universal differential graded algebras considered in \cite{connesII, karoubiI,
  karoubiII}, for example. We will find use for this construction in Section
\ref{cobordpert} where we build an explicit generalized chain which implements
the cobordism between the cycles associated with Fredholm modules which are
perturbations of each other.

Let $A = \oplus_{i=0}^\infty A_i$ be a unital graded algebra over $\cc$. We
let $\overline{A} = A/\cc$ denote the vector space quotient. Let $k \in
\nn_0$. To each multi index $I = (i_0,\ldots,i_k) \in \nn_0^{k+1}$ we
associate the $\cc$-vector space
\[
\Omega_I(A) = \fork{ccc}{
A_{i_0} \otimes_\cc \overline{A}_{i_1} \otimes_\cc \ldots \otimes_\cc
\overline{A}_{i_k}
& \T{for} & k \geq 1 \\
A_{i_0} & \T{for} & k = 0
}
\]
The degree of the elements in $\Omega_I(A)$ is given by $|I| = i_0 \plp i_k +
k$.

For each $n \in \nn_0$ we have the $\cc$-vector space $\Omega_n(A)$ defined as
the direct sum
\[
\Omega_n(A) = \bigoplus_{k=0}^n \bigoplus_{I \in \nn_0^{k+1},\, |I| = n}
\Omega_I(A)
\]
As a $\cc$-vector space the universal differential graded algebra of the
unital graded $\cc$-algebra $A$ is given by the direct sum
\[
\Omega(A) = \bigoplus_{n=0}^\infty \Omega_n(A)
\]

The differential on $\Omega(A)$ is given by 
\[
\arr{ccc}{
d : \Omega_n(A) \to \Omega_{n+1}(A) 
& \q & d(\omega_0 \olo \omega_k) = 1_A \otimes \omega_0 \olo \omega_k
}
\]
Clearly we then have $d^2 = 0$.

We endow $\Omega(A)$ with the structure of an $A$-bimodule by defining the
following products
\[
\begin{split}
& \omega \cdot (\omega_0 \olo \omega_k) 
= (\omega \omega_0) \otimes \omega_1 \olo \omega_k \qq \T{and} \\
& (\omega_0 \olo \omega_k) \cdot \omega 
= \omega_0 \olo \omega_{k-1} \otimes (\omega_k \omega) \\
& \qq + \sum_{i=0}^{k-1} 
(-1)^{|\omega_{i+1}| \plp |\omega_k| + k - i}
\omega_0 \olo (\omega_i \omega_{i+1}) \olo \omega_k \otimes \omega
\end{split}
\] 
of homogeneous elements $\omega_0 \olo \omega_k \in \Omega_n(A)$ and elements
$\omega \in A$. 

The graded multiplicative structure on $\Omega(A)$ is given by the product
\[
(\omega_0 \olo \omega_k) \cdot (\omega_0' \otimes \omega_1' \olo \omega_l') 
= \big( (\omega_0 \olo \omega_k) \cdot \omega_0' \big) \otimes \omega_1' \olo \omega_l'
\]
of homogeneous elements $\omega_0 \olo \omega_k \in \Omega_n(A)$ and
$\omega_0' \olo \omega_l' \in \Omega_m(A)$.

It can be proved by a straightforward computation that the differential $d :
\Omega_*(A) \to \Omega_{*+1}(A)$ is a graded derivation.

We let $\rho : A \to \Omega(A)$ denote the unital homomorphism of graded
algebras which is given by inclusion. Our differential graded algebra
$\Omega(A)$ satisfies a universal property, justifying its name:

\begin{prop}
Let $A = \oplus_{i=0}^\infty A_i$ be a unital graded algebra over $\cc$ and
let $\Omega = \oplus_{n=0}^\infty \Omega_n$ be a unital differential graded
algebra over $\cc$. Suppose that we have a unital homomorphism of \emph{graded
  algebras} $\phi : A \to \Omega$. Then there is a \emph{unique} unital
homomorphism of differential graded algebras $\Omega(\phi) : \Omega(A) \to
\Omega$ such that
\[ 
\Omega(\phi) \circ \rho = \phi
\]
\end{prop}
\begin{proof}
The unital homomorphism of differential graded algebras $\Omega(\phi) :
\Omega(A) \to \Omega$ is given by
\[
\Omega(\phi)(\omega_0 \olo \omega_k) =
\phi(\omega_0)d_\Omega(\phi(\omega_1)) \ldots d_\Omega(\phi(\omega_k)) 
\]
Here $d_\Omega : \Omega_* \to \Omega_{* + 1}$ denotes the differential on
$\Omega$. Note that $d_\Omega$ is a graded derivation for the multiplicative
structure on $\Omega$ by assumption.

The uniqueness of the construction is immediate.
\end{proof}





\subsection{Perturbations of Fredholm modules and
  cobordisms}\label{cobordpert}
In this section we will prove that the Chern-Connes character of $m$-summable
Fredholm modules with values in the cyclic cohomology group $HC^{m-1}(A)$ is
invariant under $m$-summable perturbations. This is a stronger version of a
result proved by A. Gorokhovsky, \cite[Remark $4.1$]{gorok}. This section also
provides the motivation for the definition of the finitely summable
$K$-homology groups which we give in Section \ref{finsumk}. Indeed, Theorem
\ref{cocyequivI} is really what ensures us, that we obtain a (non-trivial)
multiplicative pairing between algebraic $K$-theory and these new $K$-homology
type groups.

Let $A$ be an algebra over $\cc$ and let $m \in \nn$ be a positive
integer. For a Hilbert space $\C H$ we will denote the $m^{\T{th}}$ Schatten
ideal by $\C L^m(\C H)$.

\begin{definition}\label{finfred}
By an $m$-summable Fredholm module, $\C F = (\pi,H,F)$, over $A$ we will
understand the given of a separable Hilbert space $H$, an algebra homomorphism
$\pi : A \to \C L(H)$ and a bounded operator $F \in \C L(H)$ such that
\begin{enumerate}
\item $F^2 - 1 = 0$.
\item $F-F^* = 0$.
\item $[F,\pi(a)] \in \C L^m(H)$.
\end{enumerate}
for all $a \in A$. In the case where $(m-1)$ is even we will also assume the
existence of a $\zz/(2 \zz)$ grading operator $\gamma \in \C L(H)$ which
anticommutes with $F$ and commutes with all elements in $\pi(A)$. In order to
unify the notation we will use the convention that $\gamma = \E{Id}$ when
$(m-1)$ is odd.

When $A$ is a locally convex topological algebra we will say that the
$m$-summable Fredholm module $\C F = (\pi,H,F)$ is \emph{continuous} if the
maps
\[
\arr{ccc}{
a \mapsto \pi(a) \in \C L(H) & \T{and} & a \mapsto [F,\pi(a)] \in \C L^m(H)
}
\]
are continuous.
\end{definition}

Let $\C F = (\pi,H,F)$ be an $m$-summable Fredholm module over $A$. We recall
from \cite[Chapter IV]{connesII} that there is an associated index cocycle
\[
\tau_{\C F} = (\tau_{\C F},0,\ldots) \in \C B^{m-1}(A)
\]
given by
\[
\arr{ccc}{
\tau_{\C F}(x_0,\ldots, x_{m-1})
= \frac{1}{2}\T{Tr}(\gamma^m F[F,x_0]\clc [F,x_{m-1}])
& \q & x_0,\ldots,x_{m-1} \in \wit A
}
\]
Here $\wit A$ denotes the unitalization of $A$. We will refer to the class of
the index cocycle in cyclic cohomology
\[
\T{Ch}(\C F) := [\tau_{\C F}] \in HC^{m-1}(A)
\]
as the \emph{Chern-Connes character} of $\C F$.

When $A$ is a locally convex topological algebra and $\C F$ is continuous the
index cocycle determines a class
\[
\T{Ch}^{\T{cont}}(\C F) := [\tau^{\T{cont}}_{\C F}] \in HC^{m-1}_{\T{cont}}(A)
\]
in continuous cyclic cohomology. We will call this class the \emph{continuous
  Chern-Connes} character of $\C F$.

\begin{definition}\label{pertfredI}
We say that two $m$-summable Fredholm modules $\C F = (\pi, H,F)$ and $\C G =
(\pi,H,G)$ over $A$ are \emph{$m$-summable perturbations} of each other if the
difference of operators
\[
G - F \in \C L^m(H)
\]
lies in the $m^{\E{th}}$ Schatten ideal. When the dimension, $(m-1)$, is even
we will also require that the grading operators of the two Fredholm modules
agree.
\end{definition}

Let $\C F = (\pi, H, F)$ and $\C G = (\pi, H, G)$ be two $m$-summable Fredholm
modules which are $m$-summable perturbations of each other. The aim of the
present section is to prove that the Chern-Connes characters of $\C F$ and $\C
G$ determine the same class in cyclic cohomology. Thus, we will show that
\[
\T{Ch}(\C F) = [\tau_{\C F}] = [\tau_{\C G}] = \T{Ch}(\C G) \in HC^{m-1}(A)
\]
This result should be compared with \cite[Theorem $4.1$]{gorok} and
\cite[Theorem $2.11$]{nistor}.

Let $T = G-F \in \C L^m(H)$ denote the difference of $G$ and $F$. We will
constantly make use of the following identity.

\begin{lemma}\label{fundid}
\[
FT + TF + T^2 = 0
\]
\end{lemma}
\begin{proof}
We have that
\[
F^2 + T^2 + FT + TF = (F+T)^2 = 1 = F^2
\]
and the result follows immediately.
\end{proof}

Let us add an element $\tau$ of order $1$ to the unitalization $\wit A$. That
is, we let $\wit A_\tau$ denote the unital graded algebra over $\cc$ which as
a vector space is given by
\[
\wit A_\tau 
:= \bigoplus_{j=1}^\infty \big( \, \wit A \, \big)^{\otimes_\cc j}
\]
The degree of an element $x = x_0 \olo x_n \in \wit A_\tau$ is defined as $|x|
= n$. The multiplication on $\wit A_\tau$ is determined by the rule,
\[
\begin{split}
(x_0 \olo x_n) \cdot (y_0 \olo y_m) 
& =  (x_0 \olo x_{n-1}) \otimes (x_ny_0)
\otimes (y_1  \olo y_m) \\
& \qq x_0,\ldots,x_n, y_0,\ldots,y_m \in \wit A
\end{split}
\]
on simple tensors. We will use the special notation $\tau$ for the element
\[
\tau := 1_{\wit A} \otimes 1_{\wit A} \in \big(\wit A_\tau\big)_1
\]
of degree $1$.

Let us consider the unital differential graded algebra $\Omega_\tau$ defined
as the universal differential graded algebra of $\wit A_\tau$ modulo the
relation $d(\tau) + \tau^2 \sim 0$. That is, we let
\[
\Omega_\tau := \Omega(A_\tau) / \big( d(\tau) + \tau^2  \sim 0 \big)
\]
with differential induced by the differential on $\Omega(A_\tau)$. In this
respect, notice that
\[
d( d(\tau) + \tau^2 ) = d(\tau^2) = (d(\tau) + \tau^2)\tau - \tau (d(\tau) + \tau^2)
\]

We will use $\Omega_\tau$ to construct a generalized chain $\F{\Omega_{\C T}}$
over $\wit A$ which provides a cobordism between two generalized cycles
$\F{\wih{\Omega}_{\C F}}$ and $\F{\wih{\Omega}_{\C G}}$ over $\wit A$. We
shall then see in Lemma \ref{coinchernwed} that the Chern characters of these
two generalized cycles coincides with the index cocycles of $\C F$ and $\C G$
up to a constant.

Furthermore, the nature of the generalized chain $\F{\Omega_{\C T}}$ together
with the explicit formula for the associated cyclic cochain will allow us to
conclude that the difference of cyclic cocycles
\[
\T{Ch}(\F{\wih{\Omega}_{\C G}}) - \T{Ch}(\F{\wih{\Omega}_{\C F}})
\in \C B^{m-1}(\wit A)
\]
is a cyclic coboundary \emph{without any application of the periodicity
  operator}. This is thus a stronger result than what is initially provided by
the cobordism relation of Theorem \ref{cobordant}.

The proof of the desired invariance result will then follow from some
straightforward argumentation concerning the restriction cochain homomorphism
$i^* : \C B^*(\wit A) \to \C B^*(\cc)$.

Let us continue with the construction of the generalized chain $\F{\Omega_{\C
    T}}$ over $\wit A$. We start by extending the algebra homomorphism $\pi :
A \to \C L(H)$ to a unital algebra homomorphism $\pi : \Omega_\tau \to \C
L(H)$. We do this in the following way:
\begin{enumerate}
\item To begin with, we define the unital algebra homomorphism $\pi : \wit
  A_\tau \to \C L(H)$ by
\[
\arr{ccc}{
\pi : \tau \mapsto T \in \C L(H) & \T{and} & (a,\lambda) \mapsto \pi(a) +
\lambda \in \C L(H)
}
\]
\item We then define the unital algebra homomorphism $\pi : \Omega(\wit A_\tau) \to \C
  L(H)$ by
\[
\pi (\omega_0 \olo \omega_k) = \pi(\omega_0)[F, \pi(\omega_1)] \ldots [F,
\pi(\omega_k)]
\]
Here the commutators $[F, \pi(\omega_i)]$ are graded, thus
\[
[F, \pi(\omega_i)] = F \pi (\omega_i) + (-1)^{|\omega_i| + 1} \pi(\omega_i) F
\]
Note that this is well-defined even when the Hilbert space $H$ is trivially
graded.
\item Finally, it follows from the identity $FT + TF + T^2 = 0$ of Lemma
  \ref{fundid}, that the unital algebra homomorphism $\pi : \Omega(\wit
  A_\tau) \to \C L(H)$ descends to the desired unital algebra homomorphism
  $\pi : \Omega_\tau \to \C L(H)$.
\end{enumerate}

In order to define the graded trace of the generalized chain $\F{\Omega_{\C
    T}}$ of dimension $m$, we need to show that the unital algebra
homomorphism $\pi : \Omega_\tau \to \C L(H)$ maps homogeneous elements of
degree $m$ to trace class operators.

\begin{lemma}
For each element $\omega \in (\Omega_\tau)_m$ the operator $\pi(\omega) \in
\C L^1(H)$ is of trace class.
\end{lemma}
\begin{proof}
We recall that the Schatten ideals satisfy the multiplication rule
\[
\C L^{m/p}(H) \cdot \C L^{m/q}(H) \subseteq \C L^{m/(p+q)}(H)
\]
for any $p,q \in \{1,\ldots,m\}$ with $p+q \in \{1,\ldots,m\}$.

Suppose that $\sigma \in \wit A_\tau$ is a homogeneous element of order $i \in
\{0,\ldots,m\}$ of the form
\[
\sigma = x_0 \tau^{i_0} \clc x_k \tau^{i_k}
\]
for some $x_0, \ldots, x_k \in \wit A$ and $i_0 \plp i_k = i$. Since $T \in \C
L^m(H)$ by assumption we get that
\[
\pi(\sigma) 
= \wit \pi(x_0) T^{i_0} \clc \wit \pi(x_k) T^{i_k} \in \C
L^{m / i}(H)
\]
Here $\wit \pi(a,\lambda) = \pi(a) + \lambda$ for any $(a,\lambda) \in \wit
A$.

Next, assume that the order of $\sigma$ is less than or equal to
$(m-1)$. Let $j \in \{0,\ldots,k\}$. Assume that $i_j = 2l + 1$ is odd. We
evolve on the graded commutator
\[
\begin{split}
[F, \wit \pi(x_j) T^{\, i_j}]
& = [F, \wit \pi(x_j)] T^{\, i_j} + \wit \pi(x_j) (FT + TF) T^{2l} \\
& = [F, \wit \pi(x_j)] T^{\, i_j} - \wit \pi(x_j) T^{2l + 2}
\end{split}
\]
Here we have used the identity $FT + TF = T^2$ of Lemma \ref{fundid}. Assume
that $i_j = 2l$ is even. The graded commutator is then given by
\[
[F, \wit \pi(x_j) T^{\, i_j}] = [F, \wit \pi(x_j)] T^{\, i_j}
\]
In both cases we get that $[F, \wit \pi(x_j) T^{\, i_j}] \in \C L^{m/(i_j
  +1)}(H)$. Since the graded commutator $[F, \cdot \,]$ is a graded
derivation we conclude that $[F, \pi(\sigma)] \in \C L^{m/(i + 1)}(H)$.

The result of the Lemma is a consequence of these two observations and the
multiplication rule for Schatten ideals.
\end{proof}

We are now ready to define the generalized chain $\mathbf{\Omega_{\C T}}$ of
dimension $m$ over $\wit A$. It is given by the following data:
\begin{enumerate}
\item The unital graded algebras
\[
\arr{ccc}{
\Omega^*([0,1]) \grad \Omega_\tau & \T{and} &
\partial(\Omega_\tau) = \Omega_\tau \oplus \Omega_\tau
}
\]
Here $\Omega^*([0,1])$ denotes the unital graded algebra of differential forms
on the unit interval and $\grad$ denotes the graded tensor product over $\cc$
of graded algebras.

The unital surjective homomorphism of graded algebras
\[
(r_1,r_0) : \Omega^*([0,1]) \grad \Omega_\tau \to \Omega_\tau \oplus \Omega_\tau
\]
given by the restriction to $1$ and $0$ respectively.

The unital algebra homomorphism
\[
\arr{ccc}{
\rho : x \mapsto 1 \grad x \in C^\infty([0,1]) \grad \wit A
& \q & x \in \wit A
}
\]
\item The graded derivations of degree one
\[
\arr{ccc}{
\nabla_\tau = d \grad 1 + 1 \grad d
+ t \grad [\tau, \cdot\,]
& \T{and} & 
\partial \nabla_\tau = (d + [\tau, \cdot \,]) \oplus d
}
\]
with curvatures
\[
\arr{ccc}{
\theta_\tau = dt \grad \tau + t \grad d(\tau)
+ t^2 \grad \tau^2 \in (\Omega_\tau)_2
& \T{and} &
\partial\theta = (d(\tau) + (\tau)^2,0) = (0,0)
}
\]
Here the connection on $\Omega^*([0,1]) \grad \Omega_\tau$ should be
understood in the following sense
\[
\nabla_\tau(\alpha \grad \omega)
= d\alpha \grad \omega + (-1)^{|\alpha|} \alpha \grad d\omega 
+ (-1)^{|\alpha|} t \alpha \grad ( \tau \omega + (-1)^{|\omega| + 1} \omega
\tau ) 
\]   
for homogeneous elements $\alpha \in \Omega^*([0,1])$ and $\omega \in
\Omega_\tau$.
\item The graded trace $\int_\tau : (\Omega_\tau)_m \to \cc$ is given by
\[
\int_\tau(\alpha \grad \omega) = \fork{ccc}{
\frac{1}{2}\big(\int_0^1 \alpha \big)\cdot \T{Tr}(\gamma^m F\pi(d\omega)) 
& \T{for} & \alpha \in \Omega^1([0,1]) \\
0 & \T{for} & \alpha \in \Omega^0([0,1])  
}
\]
\end{enumerate}

We leave it to the reader to verify that $\F{\Omega_{\C T}}$ satisfies the
generalized chain conditions given in Section \ref{genchaincobord}. See also
\cite[IV $1.\alpha$ Proposition $1$]{connesII}, \cite[Section $4$]{gorok} and
\cite[Theorem $2.11$]{nistor}. 

Let $\F{\wih \Omega_{\C G}}$ and $\F{\wih \Omega_{\C F}}$ be the cycles of
dimension $(m-1)$ over $\wit A$ given by
\[
\arr{ccc}{
\F{\wih \Omega_{\C G}} := (\Omega_\tau, d + [\tau,\cdot\,], \int)
& \q & 
\F{\wih \Omega_{\C F}} := (\Omega_\tau, d, \int)
}
\]
Here $\int : (\Omega_\tau)_{m-1} \to \cc$ is the graded trace defined by
\[
\int \omega = \frac{1}{2}\T{Tr}(\gamma^m F \pi(d \omega))
\]
The unital algebra homomorphism $\rho : \wit A \to (\Omega_\tau)_0 = \wit A$
is the identity homomorphism. The boundary of the generalized chain
$\F{\Omega_{\C T}}$ then equals the direct sum of cycles
\[
\partial{\F{\Omega_{\C T}}} 
= \F{\wih \Omega_{\C G}} \oplus \F{\underline{\wih \Omega}_{\C F}}
\]
Here $\F{\underline{\wih \Omega}_{\C F}}$ is obtained from $\F{\wih \Omega_{\C
    F}}$ by changing the sign of the graded trace. This means that
$\F{\Omega_{\C T}}$ defines a cobordism between $\F{\wih \Omega_{\C F}}$ and
$\F{\wih \Omega_{\C G}}$. In particular we have the
cobordism relation
\begin{equation}\label{eq:scobordI}
(b+B)(\T{Ch}(\F{\Omega_{\C T}})) 
= S(\T{Ch}(\F{\wih \Omega_{\C G}}))
- S(\T{Ch}(\F{\wih \Omega_{\C F}}))
\end{equation}
at the level of cyclic cochains in $\C B^{m+1}(\wit A)$, see Theorem
\ref{cobordant}.

As mentioned earlier, the Chern characters of the cycles $\F{\wih \Omega_{\C
    G}}$ and $\F{\wih \Omega_{\C F}}$ coincides with the index cocycles of the
Fredholm modules $\C G$ and $\C F$ up to a constant.

\begin{lemma}\label{coinchernwed}
We have the identities
\[
\arr{ccc}{
(m-1)!\, \E{Ch}(\F{\wih{\Omega}_{\C F}}) = \tau_{\C F}
& \E{and} &
(m-1)!\, \E{Ch}(\F{\wih{\Omega}_{\C G}}) = \tau_{\C G}
}
\]
at the level of cyclic cocycles in $\C B^{m-1}(\wit A)$.
\end{lemma}
\begin{proof}
The first identity is clear. The proof of the second identity amounts to the
following calculation
\[
\begin{split}
& \T{Tr}\big(\gamma^m x_0[G,x_1] \ldots [G,x_{m-1}]
+ (-1)^m \gamma^m Fx_0[G,x_1]\ldots [G,x_{m-1}]F\big) \\
& \q = \T{Tr}\big(\gamma^m x_0[G,x_1] \ldots [G,x_{m-1}] 
+ (-1)^m\gamma^mGx_0[G,x_1]\ldots [G,x_{m-1}]G\big) \\
& \qq - (-1)^m\T{Tr}\big( \gamma^m Tx_0[G,x_1]\ldots [G,x_{m-1}](F+T) \big) \\
& \qq - (-1)^m\T{Tr}( \gamma^m Fx_0[G,x_1]\ldots [G,x_{m-1}]T) \\
& \q = \T{Tr}(\gamma^m G[G,x_0]\clc [G,x_{m-1}]) \\
\end{split}
\]
which is valid for all $x_0,\ldots,x_{m-1} \in \wit A$. Here we have used that
the operator $FT + TF +T^2$ vanishes, see Lemma \ref{fundid}.
\end{proof}

The result of Lemma \ref{coinchernwed} improves the equality of
\eqref{eq:scobordI}, which now reads
\begin{equation}\label{eq:scobordII}
(b+B)(\T{Ch}(\F{\Omega_{\C T}})) 
= \frac{1}{(m-1)!}S(\tau_{\C G} - \tau_{\C F})
\end{equation}
The explicit description of the involved cyclic cochains allows us to carry
out a further improvement of the above relation. In this respect, we should
refer to the proof of \cite[Lemma $2.7$]{gorok} where the main idea of the
next Lemma appears.

\begin{lemma}\label{diffunI}
The difference of cyclic cocycles
\[
\tau_{\C G} - \tau_{\C F}
\in \C B^{m-1}(\wit A)
\]
is a cyclic coboundary.
\end{lemma}
\begin{proof}
Let $(y_{m+1},y_{m-1},\ldots ) \in \C B_{m+1}(\wit A)$. From the identity
\eqref{eq:scobordII} we get that
\[
\begin{split}
& \tau_{\C G}(y_{m-1},y_{m-3},\ldots) - \tau_{\C F}(y_{m-1},y_{m-3},\ldots) \\
& \q = (m-1)! \, \T{Ch}(\F{\Omega_{\C T}})(by_{m+1} + By_{m-1}, by_{m-1} +
By_{m-3},\ldots)
\end{split}
\]
Now, let $\T{Ch}(\F{\Omega_{\C T}})^m$ denote the top component of
$\T{Ch}(\F{\Omega_{\C T}})$. We then have
\[
\T{Ch}(\F{\Omega_{\C T}})^m(x_0,\ldots,x_m)
= \frac{1}{m!}\int_\tau (1 \grad x_0) \nabla_\tau(1 \grad x_1)
\ldots \nabla_\tau(1 \grad x_m)
\]
Here $x_0,\ldots,x_m \in \wit A$ are any elements in the
unitalization. However by definition of the graded trace
\[
\int_\tau : (\Omega^*([0,1]) \grad \Omega_\tau)_m \to \cc
\]
the above quantity is zero, since
\[
(1 \grad x_0) \nabla_\tau(1 \grad x_1) \ldots \nabla_\tau(1 \grad x_m) \in
\Omega^0([0,1]) \grad (\Omega_\tau)_m
\]
It follows that 
\[
\begin{split}
& \tau_{\C G}(y_{m-1},y_{m-3},\ldots) - \tau_{\C F}(y_{m-1},y_{m-3},\ldots) \\
& \q = (m-1)!\, \T{Ch}(\F{\Omega_{\C T}})^{m-2}(by_{m-1}+By_{m-3})
+ (m-1)!\, \T{Ch}(\F{\Omega_{\C T}})^{m-4}(by_{m-3}+By_{m-5})
+ \ldots 
\end{split}
\]
But this proves the lemma.
\end{proof}

We are now in position to prove the invariance result which lies at the core
of the present paper.

\begin{prop}\label{cocyequivI}
Let $\C F = (\pi,H,F)$ and $\C G = (\pi,H,G)$ be two $m$-summable Fredholm
modules over a $\cc$-algebra $A$. Suppose that $\C F$ and $\C G$ are
$m$-summable perturbations of each other. Then their Chern-Connes characters
\[
\E{Ch}(\C F)
= \E{Ch}(\C G) \in HC^{m-1}(A)
\]
agree in the cyclic cohomology group of degree $(m-1)$.
\end{prop}
\begin{proof}
Let $i^* : \C B^*(\wit A) \to \C B^*(\cc)$ denote the cochain homomorphism
induced by the inclusion $i : \cc \to \wit A$.

Suppose that $m = 2k + 1$ is odd. By Lemma \ref{diffunI} we have the identity
\[
(b+B)\big(\T{Ch}(\Omega_{\C T})^{2k-1},\ldots,\T{Ch}(\Omega_{\C T})^1\big)
= \frac{1}{(m-1)!}\tau_{\C G} - \frac{1}{(m-1)!}\tau_{\C F}
\]
in $\C B^{2k}(\wit A)$. The desired result follows by noting that all the
involved cyclic cochains lie in the kernel of the restriction $i^*$.

Let $p : \wit A \to \cc$ denote the unital algebra homomorphism given by
$p(a,\lambda) = \lambda$.

Suppose that $m = 2k$ is even. Let $\psi \in \C B^{m-2}(\wit A)$ denote the
cyclic cochain given by
\[
\psi(y_{2k-2},\ldots,y_0)
= \big(\T{Ch}(\Omega_{\C T})^{2k-2},\ldots,\T{Ch}(\Omega_{\C   T})^0\big)
(y_{2k-2},\ldots,y_0)
- \T{Ch}(\Omega_{\C T})^0\big(p(y_0)\big)
\]
We then have that $i^*(\psi) = 0$. Furthermore, the coboundary of $\psi$ is
the desired difference of Chern characters
\[
(b + B)(\psi) 
= (b+B)\big(\T{Ch}(\Omega_{\C T})^{2k-2},\ldots,\T{Ch}(\Omega_{\C T})^0\big)
= \frac{1}{(m-1)!}\tau_{\C G} - \frac{1}{(m-1)!}\tau_{\C F}
\]
This proves the theorem in the even case as well.
\end{proof}

On some occasions in the sequel we will also need a continuous version of
Theorem \ref{cocyequivI}.

\begin{prop}\label{cocyequivcontI}
Let $\C F = (\pi, H, F)$ and $\C G = (\pi, H, G)$ be two continuous
$m$-summable Fredholm modules over $A$. Assume that $\C F$ and $\C G$ are
$m$-summable perturbations of each other. Then their continuous Chern-Connes
characters
\[
\E{Ch}^{\E{cont}}(\C F)
= \E{Ch}^{\E{cont}}(\C G) \in
HC^{m-1}_{\E{cont}}(A)
\]
agree in the continuous cyclic cohomology group of dimension $(m-1)$.
\end{prop}
\begin{proof}
This follows by noting that the cyclic cochain $\T{Ch}(\Omega_{\C T}) \in \C
B^m(\wit A)$, which is used for the proof of Theorem \ref{cocyequivI},
satisfies the relevant continuity properties.
\end{proof}





\section{Finitely summable $K$-homology and pairings with algebraic
  $K$-theory}\label{finsumk}
We will define a variant of analytic $K$-homology which has a more algebraic
flavour. The main idea is simply to replace the compact operators in the
definition of analytic $K$-homology with Schatten ideals. The most "algebraic"
equivalence relations in the $C^*$-algebraic setup are unitary equivalence and
perturbation by compact operators. We will thus replace the usual definitions
with the corresponding finitely summable versions. Our main result is then
that the new $K$-homology type groups which we obtain actually pair with
algebraic $K$-theory by means of the Connes-Karoubi multiplicative character.

The results in this section rely on the invariance under perturbations of the
Chern character which we proved in Section \ref{invarcyc}.

The material is organized as follows:

In subsection \ref{finfredeq} we introduce the finitely summable $K$-homology
together with natural periodicity and comparison homomorphisms. The
definitions also make sense in a continuous setup.

In subsection \ref{addcharpairing} we show that the continuous version of
finitely summable $K$-homology pairs with M. Karoubi's relative
$K$-groups. The pairing is induced by the "additive" character of continuous
finitely summable Fredholm modules.

In subsection \ref{multcharpairing} we show that finitely summable
$K$-homology pairs with algebraic $K$-theory. The pairing is induced by the
Connes-Karoubi multiplicative character of finitely summable Fredholm
modules. We review the construction of the multiplicative character also.

\subsection{Finitely summable Fredholm modules and their equivalence
  relations}\label{finfredeq}
In this subsection we will introduce the finitely summable $K$-homology groups
of an algebra over the complex numbers. These groups are equipped with
periodicity homomorphisms which correspond to the periodicity homomorphisms in
cyclic cohomology. Furthermore, for pre-$C^*$-algebras, there are interesting
comparison homomorphisms from finitely summable $K$-homology to analytic
$K$-homology. To begin with, let us give some relevant definitions concerning
finitely summable Fredholm modules.

Let $A$ be an algebra over $\cc$ and let $m \in \nn$ be a positive
integer. Let $\C F = (\pi,H,F)$ be an $m$-summable Fredholm module over $A$.

We will make the following standard assumption: When $(m-1)$ is even we will
assume that the eigenspaces of the grading operator are infinite
dimensional. When $(m-1)$ is odd we will always assume that the eigenspaces of
the selfadjoint unitary $F$ are infinite dimensional. This mild condition is
necessary for the construction in Subsection \ref{multcharpairing} of the
Connes-Karoubi multiplicative pairing between finitely summable $K$-homology
and algebraic $K$-theory. See also \cite[\S $1$]{conkar}.

\begin{definition}\label{degfredI}
We say that $\C F$ is \emph{degenerate} if all the relations in Definition
\ref{finfred} are exactly satisfied. That is, they should not only be
satisfied modulo the $m^{\T{th}}$ Schatten ideal.
\end{definition}

We let $\C F^{-1}$ denote the $m$-summable Fredholm module over $A$ given by
\begin{equation}\label{eq:invfred}
\C F^{-1} := \fork{ccc}{
(\pi,H,-F) & \T{for} & (m-1) \T{ odd} \\
(\pi,H^{\T{op}},-F) & \T{for} & (m-1) \T{ even}
}
\end{equation}
Here $H^{\T{op}}$ denotes the Hilbert space with grading operator $-\gamma \in
\C L(H)$ whenever $H$ is a $\zz/(2 \zz)$-graded Hilbert space with grading
operator $\gamma \in \C L(H)$. 

We will think of $\C F^{-1}$ as a specific choice of an inverse to $\C
F$. This choice is motivated by the standard choice of an inverse in
$K$-homology, see \cite[Proposition $17.3.3$]{black} and \cite[Proposition
$8.2.10$]{higson} for example. It should also be noted that the Chern-Connes
character of the inverse is given by minus the Chern-Connes character of the
original element.

Let $\C F_1 = (\pi_1,H_1,F_1)$ and $\C F_2 = (\pi_2,H_2,F_2)$ be two
$m$-summable Fredholm modules over $A$.
 
\begin{definition}\label{directsumI}
By the \emph{direct sum} of $\C F_1$ and $\C F_2$ we will understand the
$m$-summable Fredholm module
\[
\C F_1 \oplus \C F_2 = (\pi_1 \oplus \pi_2, H_1 \oplus H_2, F_1 \oplus F_2)
\]
over $A$.

In the case where $(m-1)$ is even, the grading on $H_1 \oplus H_2$ is given by
the direct sum of the grading operators.
\end{definition}

We will now concentrate on the appropriate equivalence relations.

\begin{definition}
We say that $\C F_1$ and $\C F_2$ are \emph{unitarily equivalent} if there
exists a unitary operator $u : H_1 \to H_2$ such that
\[
\arr{ccc}{
\pi_1 = u^* \pi_2 u & \T{and} & F_1 = u^* F_2 u
}
\]
In the case where $(m-1)$ is even we will also require that $\gamma_1 = u^*\gamma_2 u$.
\end{definition}

The notation $\sim_m$ will refer to the equivalence relation on $m$-summable
Fredholm modules generated by unitary equivalence and $m$-summable
perturbations in the sense of Definition \ref{pertfredI}.

\begin{definition}\label{stabym}
We say that $\C F_1$ and $\C F_2$ are \emph{stable $m$-summable perturbations}
of each other when there exist $m$-summable Fredholm modules $\C G_1, \C G_2$
and $\C H$ together with degenerate $m$-summable Fredholm modules $\C D_1$ and
$\C D_2$ such that
\[
\C F_1 \oplus \C G_1 \oplus \C G_1^{-1} \oplus \C D_1 \oplus \C H
\sim_m \C F_2 \oplus \C G_2 \oplus \C G_2^{-1} \oplus \C H
\]
We will denote the equivalence relation of stable $m$-summable perturbations
by $\sim_{mc}$.
\end{definition}

We are now ready to define the finitely summable $K$-homology of a
$\cc$-algebra. This is the main definition of this section.

\begin{definition}
By the \emph{$m$-summable $K$-homology} of $A$ we will understand the abelian
group given by the $m$-summable Fredholm modules over $A$ modulo the
equivalence relation of stable $m$-summable perturbation.

The group operation is given by the direct sum of $m$-summable Fredholm
modules. The $m$-summable $K$-homology of $A$ is denoted by $FK^{m-1}(A)$.
\end{definition}

Let $A$ be a locally convex topological algebra. The above definitions apply
to continuous Fredholm modules over $A$ as well. This gives rise to a
continuous version of finitely summable $K$-homology.

\begin{definition}
By the \emph{continuous $m$-summable $K$-homology} of $A$ we will understand
the abelian group given by the $m$-summable continuous Fredholm modules over
$A$ modulo the equivalence relation of stable $m$-summable perturbation.

The group operation is given by the direct sum of $m$-summable Fredholm
modules. The continuous $m$-summable $K$-homology of $A$ is denoted by
$FK^{m-1}_{\E{cont}}(A)$.
\end{definition}

Let $A$ and $B$ be algebras over $\cc$ and let $\varphi : A \to B$ be an
algebra homomorphism. To each $m$-summable Fredholm module, $\C F =
(\pi,H,F)$, over $A$ we can associate an $m$-summable Fredholm module over
$B$,
\[
FK(\varphi)(\C F) := (\pi \circ \varphi, H, F)
\]
In this way we obtain a group homomorphism $FK(\varphi) : FK^{m-1}(A) \to
FK^{m-1}(B)$ for each $m \in \nn$. This turns finitely summable $K$-homology
into a contravariant functor from the category of algebras over $\cc$ to the
category of abelian groups.

Likewise, the continuous finitely summable $K$-homology becomes a
contravariant functor from the category of locally convex topological algebras
to the category of abelian groups.

Let us describe some interesting forgetful homomorphisms.

Let $\C F$ be an $m$-summable Fredholm module over a $\cc$-algebra $A$. Since
the $m^{\T{th}}$ Schatten ideal is contained in the $(m+2)^{\T{th}}$ Schatten
ideal, $\C F$ is also an $(m+2)$-summable Fredholm module. This forgetful map
induces a natural periodicity homomorphism
\[
S : FK^{m-1}(A) \to FK^{m+1}(A)
\]
The construction applies in the continuous case as well. Thus, for each locally
convex topological algebra, $A$, and each $m \in \nn$ there is a natural
periodicity homomorphism
\[
S : FK^{m-1}_{\T{cont}}(A) \to FK_{\T{cont}}^{m+1}(A)
\]

Let $A$ be a pre-$C^*$-algebra in the sense of \cite[IV $1.\gamma$]{connesII}
and let $\C F = (\pi, H,F)$ be an $m$-summable Fredholm module over $A$. The
algebra homomorphism $\pi : A \to \C L(H)$ extends uniquely to a continuous
algebra homomorphism $\overline{\pi} : \overline{A} \to \C L(H)$. Here
$\overline{A}$ denotes the $C^*$-algebra closure of $A$ and the topology on
$\C L(H)$ is defined by the operator norm. Since the compact operators agrees
with the closure of the $m^{\T{th}}$ Schatten ideal in $\C L(H)$ we get a
Fredholm module $(\overline{\pi},H,F)$ over $\overline{A}$. This operation
induces a comparison homomorphism
\[
\arr{ccc}{
\alpha : FK^{m-1}(A) \to K^j(\overline A) & \q & j \T{ parity as }m-1
}
\]
for each $m \in \nn$. Here $K^*(\overline A)$ denotes the analytic
$K$-homology of the $C^*$-algebra $\overline A$.

\subsection{The additive character and pairings with relative
  $K$-theory}\label{addcharpairing}
In this subsection we will show that continuous finitely summable $K$-homology
pairs with M. Karoubi's relative $K$-theory. It was already noted in
\cite{conkar} that each continuous finitely summable Fredholm module yields an
"additive" character on relative $K$-theory. We use the invariance under
perturbations of the Chern-Connes character of continuous finitely summable
Fredholm modules to establish invariance properties of this additive
character. To begin with, let us show that finitely summable $K$-homology
pairs with cyclic homology.

Let $A$ be an algebra over $\cc$ and let $m \in \nn$ be a positive
integer. Let $\C F$ be an $m$-summable Fredholm module over $A$.

We recall from Section \ref{cobordpert} that $\T{Ch}(\C F) = [\tau_{\C F}]$
denotes the class of the index cocycle in cyclic cohomology.
  
Let $\C G$ be another $m$-summable Fredholm module over $A$ and let $\C D$ be
an $m$-summable degenerate Fredholm module over $A$. Recall from Section
\ref{finfredeq} that $\C F^{-1}$ denotes the inverse Fredholm module of $\C
F$. We have the following well known properties
\begin{align}
& \T{Ch}(\C F \oplus \C G) = \T{Ch}(\C F) + \T{Ch}(\C G) \\
& \T{Ch}(\C D) = 0 \\
& \T{Ch}(\C F) + \T{Ch}(\C F^{-1}) = 0 \\
& \T{Ch}(\C F) = \T{Ch}(\C G) \qqqq \T{whenever} \q \C F \sim_u \C G
\end{align}
for the Chern-Connes characters. Together with the invariance result of
Theorem \ref{cocyequivI} these properties imply that there is a well defined
homomorphism
\[
\arr{ccc}{
\T{Ch} : FK^{m-1}(A) \to HC^{m-1}(A) & \q & \T{Ch}([\C F]) = \T{Ch}(\C F)
}
\]
In particular we get a pairing
\[
\arr{ccc}{
\tau : FK^{m-1}(A) \times HC_{m-1}(A) \to \cc
& \q & ([\C F],[x]) \mapsto \tau_{\C F}([x])
}
\]
between finitely summable $K$-homology and cyclic homology.

Let $A$ be a locally convex topological algebra and let $\C F$ be a continuous
$m$-summable Fredholm module over $A$.

We recall from Section \ref{cobordpert} that $\T{Ch}^{\T{cont}}(\C F) =
[\tau^{\T{cont}}_{\C F}] \in HC^{m-1}_{\T{cont}}(A)$ denotes the class in
continuous cyclic cohomology of the index cocycle. For the same reasons as
above we get a well defined homomorphism
\[
\arr{ccc}{
\T{Ch}^{\T{cont}} : FK^{m-1}_{\T{cont}}(A) \to HC^{m-1}_{\T{cont}}(A)
& \q & \T{Ch}^{\T{cont}}([\C F]) = \T{Ch}^{\T{cont}}(\C F)
}
\]
in the continuous case as well. See Theorem \ref{cocyequivcontI}. In
particular, we have a pairing
\begin{equation}\label{eq:paircychom}
\arr{ccc}{
\tau^{\T{cont}} : FK^{m-1}_{\T{cont}}(A) \times HC_{m-1}^{\T{cont}}(A) \to \cc
& \q & ([\C F],[x]) \mapsto \tau^{\T{cont}}_{\C F}([x])
}
\end{equation}
between continuous finitely summable $K$-homology and continuous cyclic
homology. Here $\tau_{\C F}^{\T{cont}} : HC_{m-1}^{\T{cont}}(A) \to \cc$
denotes the character on continuous cyclic homology induced by the class
$\T{Ch}^{\T{cont}}(\C F) \in HC^{m-1}(A)$ in continuous cyclic cohomology.

Remark that both of the Chern-Connes characters which we have defined above
are well-behaved with respect to the periodicity operators in (continuous)
finitely summable $K$-homology and (continuous) cyclic cohomology. This
follows by \cite[IV $1.\beta$ Proposition $2$]{connesII}.

Let us proceed with some short reminders on relative $K$-theory and the
additive character.

In the book \cite{karoubiII} M. Karoubi constructs a sequence of covariant
functors from the category of Fréchet algebras to the category of abelian
groups. For each positive integer $m \in \nn$ and each Fréchet algebra $A$ the
associated abelian group is denoted by $K_m^{\T{rel}}(A)$ and termed the $m^{\T{th}}$
\emph{relative $K$-group} of $A$. These abelian groups measure the difference
between algebraic and topological $K$-theory.

\begin{prop}\label{longrelI}
For each Fréchet algebra $A$ there is a long exact sequence
\[
\begin{CD}
\ldots @>\partial>> K_m^{\E{rel}}(A) @>\theta>>
K_m(A) @>>> K_m^{\E{top}}(A) @>\partial>> K^{\E{rel}}_{m-1}(A) @>>> \ldots
\end{CD}
\]
Here the notation $K_*(A)$ and $K_*^{\T{top}}(A)$ refers to the algebraic and
topological $K$-theory of the Fréchet algebra, respectively.
\end{prop}

The relative $K$-theory is linked to continuous cyclic homology by means of a
relative Chern character.

\begin{prop}\label{chrel}\cite{conkar,karoubiII}
For each Fréchet algebra $A$ there is a natural homomorphism of
degree minus one
\[
\E{ch}^{\E{rel}} : K_m^{\E{rel}}(A) \to HC^{\E{cont}}_{m-1}(A)
\]
from relative $K$-theory to continuous cyclic homology. This homomorphism is
termed \emph{the relative Chern character}.
\end{prop}

In the above definition the relative Chern character is normalized as in
\cite{kaadII}. That is, we have
\[
\T{ch}^{\T{rel}} 
:= (-1)^m (m-1)! \, \underline{\T{ch}}^{\T{rel}}
: K_m^{\T{rel}}(A) \to HC^{\, \T{cont}}_{m-1}(A)
\]
where $\underline{\T{ch}}^{\T{rel}}$ denotes the Chern character as defined in
\cite[\S $4$]{conkar}. The additional constant ensures that the relative Chern
character is both multiplicative in an appropriate sense and induced by a
chain map. See also \cite{till}.

Let $\C F$ be a continuous $m$-summable Fredholm module over $A$. We can
compose the relative Chern character with the character on continuous cyclic
homology induced by $\C F$. In this way we obtain the \emph{additive
  character} on relative $K$-theory,
\[
\arr{ccc}{
\C A_{\C F} : K_m^{\T{rel}}(A) \to \cc
& \q & \C A_{\C F} 
= c_{m-1} \cdot (\tau^{\T{cont}}_{\C F} \circ \T{ch}^{\T{rel}})
}
\]
Here $c_{m-1} \in \Qq$ is the rational constant given by
\[
c_{m-1} = \fork{ccc}{
(-1)^{k+1} \frac{k!}{(2k)!} & \T{ for } & m-1 = 2k \T{ even} \\
- \frac{1}{2^{2k-1}\cdot (k-1)!} & \T{ for } & m-1 = 2k-1 \T{ odd}
}
\]
This constant is chosen in such a way that the multiplicative character which
we will define in the next section agrees with the multiplicative character
defined in \cite{conkar}.

The existence of the pairing \eqref{eq:paircychom} implies the following
result,

\begin{prop}\label{pairaddi}
For each Fréchet algebra $A$ the map
\[
\arr{ccc}{
\C A : FK^{m-1}_{\E{cont}}(A) \times K_m^{\E{rel}}(A) \to \cc
& \q & \C A([\C F],[\sigma]) \mapsto \C A_{\C F}([\sigma])
}
\]
yields a well-defined pairing between continuous finitely summable
$K$-homology and relative $K$-theory.
\end{prop}





\subsection{The multiplicative character and pairings with algebraic
  $K$-theory}\label{multcharpairing}
In this subsection we will show that the finitely summable $K$-homology groups
pair with algebraic $K$-theory. It was already noted by A. Connes and
M. Karoubi in \cite{conkar} that each finitely summable Fredholm module gives
rise to a multiplicative character on algebraic $K$-theory. An important
ingredient in their construction is the additive character which we considered
in the last section. The well-definedness of the pairing given in Theorem
\ref{pairaddi} will therefore aid us in defining our pairing between finitely
summable $K$-homology and algebraic $K$-theory. Let us start by giving some
details on the Connes-Karoubi multiplicative character. To this end, we will
need the concept of "universal Fredholm algebras".

Let $\C H$ be a fixed separable infinite dimensional Hilbert space and let $m
\in \nn$ be a positive integer. 

Suppose that $(m-1)$ is even. We let $\C M^{m-1} \subseteq \C L(\C H \oplus \C
H)$ denote the $\cc$-subalgebra of diagonal operators, $d(x,y)$, with
difference $x-y \in \C L^m(\C H)$ in the $m^{\T{th}}$ Schatten ideal. This
$\cc$-algebra becomes a unital Banach algebra when equipped with the norm
\[
\arr{ccc}{
\|x\| = \|x\|_\infty + \|[F_{m-1},x]\|_m & \q & x \in \C M^{m-1}
}
\]
Here $F_{m-1}$ denotes the bounded operator
\[
F_{m-1} = \mat{cc}{0 & 1 \\ 1 & 0} \in \C L(\C H \oplus \C H)
\]
The notations $\|\cdot\|_\infty$ and $\|\cdot\|_m$ refer to the operator norm
and the norm on the $m^{\T{th}}$ Schatten ideal.

Suppose that $(m-1)$ is odd. We let $\C M^{m-1} \subseteq \C L(\C H \oplus \C
H)$ denote the $\cc$-subalgebra of operators with antidiagonal in the
$m^{\T{th}}$ Schatten ideal. This $\cc$-algebra becomes a unital Banach
algebra when equipped with the norm
\[
\arr{ccc}{
\|x\| = \|x\|_\infty + \|[F_{m-1},x]\|_m & \q & x \in \C M^{m-1}
}
\]
Here $F_{m-1}$ denotes the bounded operator
\[
F_{m-1} = \mat{cc}{1 & 0 \\ 0 & -1} \in \C L(\C H \oplus \C H)
\]

In both the even and the odd case there is a universal $m$-summable Fredholm
module over $\C M^{m-1}$ given by
\[
\C F_{m-1} = (i,\C H \oplus \C H, F_{m-1})
\]
Here $i : \C M^{m-1} \to \C L(\C H \oplus \C H)$ is the inclusion. The grading
operator is given by the diagonal matrix $\gamma = d(1,-1)$ when $(m-1)$ is
even. Remark that the universal $m$-summable Fredholm module is continuous in
the sense of Definition \ref{finfred}.

The notation
\[
\C A_{m-1} : K_m^{\T{rel}}(\C M^{m-1}) \to \cc
\]
will refer to the additive character induced by the universal $m$-summable
Fredholm module. See Section \ref{addcharpairing}.

The topological $K$-groups of the unital Banach algebras $\C M^{m-1}$ are
computed by A. Connes and M. Karoubi.

\begin{lemma}\cite[Théorème $2.8$]{conkar}
For any $m \in \nn$ the topological $K$-theory of $\C M^{m-1}$ is given by
\[
\begin{split}
K_0^{\E{top}}(\C M^{m-1}) & = \fork{ccc}{
\{0\} & \T{for} & (m-1) \T{ odd} \\
\zz & \T{for} & (m-1) \T{ even}
} \\
K_1^{\E{top}}(\C M^{m-1}) & = \fork{ccc}{
\zz & \T{for} & (m-1) \T{ odd} \\
\{0\} & \T{for} & (m-1) \T{ even} \\
}
\end{split}
\]
\end{lemma}

In the case of the unital Banach algebra $\C M^{m-1}$ the long exact sequence
of $K$-groups from Theorem \ref{longrelI} then reads
\[
\begin{CD}
\ldots  @>\theta>> K_{m+1}(\C M^{m-1}) @>>> \zz  @>\partial>> 
K_m^{\T{rel}}(\C M^{m-1}) @>\theta>> K_m(\C M^{m-1}) @>>> 0
\end{CD}
\]
In particular, the homomorphism $\theta : K_m^{\T{rel}}(\C M^{m-1}) \to
K_m(\C M^{m-1})$ is surjective. Furthermore, since both the boundary map
and the universal additive character are homomorphisms we get that
\[
\T{Im}(\C A_{m-1} \circ \partial) = c \zz \subseteq \cc
\]
where $c \in \cc$ is some constant. In \cite{conkar} this constant is
calculated to be $c = (2 \pi i)^{\lceil m/2 \rceil}$. These results allow us
to define the universal multiplicative character.

\begin{definition}
By the \emph{universal $m$-summable multiplicative character} we will
understand the homomorphism
\[
\C M_{m-1} : K_m(\C M^{m-1}) \to \cc/(2 \pi i)^{\lceil m/2 \rceil} \zz
\]
given by
\[
\C M_{m-1}( \theta(y)) = [\C A_{m-1}(y)]
\]
Here $[\, \cdot \, ] : \cc \to \cc/(2 \pi i)^{\lceil m/2 \rceil} \zz$ denotes
the quotient map.
\end{definition}

Let $A$ be an algebra over $\cc$ and let $\C F$ be an $m$-summable Fredholm
module over $A$.

Suppose that $(m-1)$ is odd and let $P = (F+1)/2$. Since the Hilbert space
$PH$ and $(1-P)H$ are infinite dimensional by assumption, we can identify both
of them with the "universal" Hilbert space $\C H$. In this way we obtain an
algebra homomorphism
\[
\pi_{\C F} : A \to \C L(\C H \oplus \C H)
\]
which is well defined up to conjugation by a diagonal unitary operator. 

Suppose that $(m-1)$ is even. By assumption both of the eigenspaces $H^+$ and
$H^-$ for the grading operator $\gamma \in \C L(H)$ are infinite
dimensional. We can thus identify each of them with the "universal" Hilbert
space $\C H$. In this way we obtain a unital algebra homomorphism
\[
\pi_{\C F} : A \to \C L(\C H \oplus \C H)
\]
which is well defined up to conjugation by a diagonal unitary operator.

In both the even and the odd case the condition on the commutator $[F,\pi(a)]
\in \C L^m(H)$ ensures that the algebra homomorphism $\pi_{\C F}$ factorizes
through the Banach algebra $\C M^{m-1}$. The algebra homomorphism $A \to \C
M^{m-1}$ will also be denoted by $\pi_{\C F}$. Remark that this algebra
homomorphism is continuous whenever $A$ is a locally convex topological
algebra and $\C F$ is continuous.

\begin{remark}
The terminology "universal", which we have used for the $m$-summable Fredholm
module $\C F_{m-1}$, is justified by the following identity in $FK^{m-1}(A)$,
\[
FK(\pi_{\C F})([\C F_{m-1}]) = [\C F]
\]
There is a similar identity in the continuous case as well.
\end{remark}

We are now in position to give the general definition of the multiplicative
character.

\begin{definition}
By the \emph{multiplicative character} of $\C F$ we will understand the
homomorphism
\[
\arr{ccc}{
\C M_{\C F} : K_m(A) \to \cc/(2 \pi i)^{\lceil m/2 \rceil}\zz 
& \q & \C M_{\C F} = \C M_{m-1} \circ (\pi_{\C F})_*
}
\]
Here $(\pi_{\C F})_* : K_m(A) \to K_m(\C M^{m-1})$ denotes the homomorphism on
algebraic $K$-theory induced by the algebra homomorphism $\pi_{\C F} : A \to
\C M^{m-1}$.
\end{definition}

Let $\C F$ and $\C G$ be two $m$-summable Fredholm modules over $A$ and let
$[x] \in K_m(A)$. The following properties for the multiplicative character
are proved in \cite{conkar},
\begin{equation}
\begin{split}
& \C M_{\C F \oplus \C G}([x]) = \C M_{\C F}([x]) + \C M_{\C G}([x]) \\
& \C M_{\C F^{-1}}([x]) = - \C M_{\C F}([x]) \\
& \C M_{\C F}([x]) = \C M_{\C G}([x]) \qqqq \T{whenever } \C F \sim_u \C G
\end{split}
\end{equation}

We would like to show that the multiplicative character is invariant under
stable $m$-summable perturbations. See Definition \ref{stabym}. Indeed, this
would ensure us that the map
\[
\arr{ccc}{
\C M : FK^{m-1}(A) \times K_m(A) \to \cc/(2 \pi i)^{\lceil m/2 \rceil} \zz
& \q & \C M([\C F],[x]) = \C M_{\C F}([x])
}
\]
yields a well-defined pairing between finitely summable $K$-homology and
algebraic $K$-theory.

In view of the above relations we will only need to consider the behaviour of
the multiplicative character with respect to finitely summable perturbations
and degeneracies. The case of degeneracies is carried out by the next lemma.

\begin{lemma}\label{degvan}
Suppose that $\C D$ is an $m$-summable degenerate Fredholm module over
$A$. Then the homomorphism of abelian groups
\[
(\pi_{\C D})_* : K_m(A) \to K_m(\C M^{m-1})
\]
vanishes identically. In particular, the multiplicative character associated
to $\C D$ is trivial.
\end{lemma}
\begin{proof}
Suppose that $(m-1)$ is odd. The induced algebra homomorphism
\[
\pi_{\C D} : A \to \C M^{m-1}
\]
is then diagonal. In particular it factorizes through the direct sum of
bounded operators $\C L(\C H) \oplus \C L(\C H)$. But this unital ring has
trivial algebraic $K$-theory by \cite{macduff}.

In the even case the induced algebra homomorphism factorizes through $\C L(\C
H)$ so the same argument applies.
\end{proof}

Suppose that we have two $m$-summable Fredholm modules $\C F_1 = (\pi,H,F_1)$
and $\C F_2 = (\pi,H,F_2)$ over $A$ which are $m$-summable perturbations of
each other. In order to show that their multiplicative character coincide we
spell out the effect of the perturbation at the level of the associated
algebra homomorphisms $\pi_{\C F_1} : A \to \C M^{m-1}$ and $\pi_{\C F_2} : A
\to \C M^{m-1}$.

\begin{lemma}\label{msummeff}
There exists a continuous $m$-summable Fredholm module $\C T$ over $\C
M^{m-1}$ such that
\begin{enumerate}
\item The diagram of algebra homomorphisms
\begin{equation}\label{eq:diansum}
\begin{CD}
A @>\E{Id}>> A \\
@V\pi_{\C F_1}VV @V\pi_{\C F_2}VV \\
\C M^{m-1} @>\pi_{\C T}>> \C M^{m-1}
\end{CD}
\end{equation}
is commutative.
\item The continuous $m$-summable Fredholm module $\C T$ is an $m$-summable
  perturbation of the universal $m$-summable Fredholm module.
\end{enumerate}
\end{lemma}
\begin{proof}
By assumption the difference $F_2-F_1 = T \in \C L^m(H)$ lies in the
$m^{\T{th}}$ Schatten ideal.

Suppose that $(m-1)$ is odd. Let $P_1 = \frac{F_1+1}{2}$ and let $P_2 =
\frac{F_2+1}{2}$. Let $u_1 : P_1H \oplus (1-P_1)H \to \C H \oplus \C H$ and
$u_2 : P_2H \oplus (1-P_2)H \to \C H \oplus \C H$ denote some diagonal unitary
operators. The algebra homomorphisms
\[
\arr{ccc}{
\pi_{\C F_1} : A \to \C M^{m-1} & \T{and} & \pi_{\C F_2} : A \to \C M^{m-1}
}
\]
are then given by
\[
\arr{ccc}{
\pi_{\C F_1} = u \pi u^* & \T{and} & \pi_{\C F_2} = v \pi v^*
}
\]
We define the continuous $m$-summable Fredholm module $\C T$ over $\C M^{m-1}$
by $\C T := (i, \C H \oplus \C H, uF_2u^* )$. Clearly $\C T$ is an
$m$-summable perturbation of the universal $m$-summable Fredholm
module. Furthermore, the algebra homomorphism associated to $\C T$ is given by
\[
\arr{ccc}{
\pi_{\C T} : \C M^{m-1} \to \C M^{m-1} & \q & \pi_{\C T}(x) = v u^* x u v^*
}
\]
This proves the odd case of the lemma. The even case follows by similar
considerations.
\end{proof}

We are now in position to prove the main result of this paper.

\begin{prop}
For each positive integer $m \in \nn$ there is a well-defined pairing
\[
FK^{m-1}(A) \times K_m(A) \to \cc/(2 \pi i)^{\lceil m/2 \rceil}\zz
\]
given by the formula
\[
\arr{ccc}{
([\C F],[x]) \mapsto \C M_{\C F}([x]) & \q & [\C F] \in FK^{m-1}(A), \, [x]
\in K_m(A)
}
\]
\end{prop}
\begin{proof}
As mentioned earlier, we will only need to prove that the multiplicative
character is invariant under $m$-summable perturbations. Thus, let $\C F_1$
and $\C F_2$ be two $m$-summable Fredholm modules over $A$ which are
$m$-summable perturbations of each other. By Lemma \ref{msummeff} we have a
continuous $m$-summable Fredholm module $\C T$ which is an $m$-summable
perturbation of the universal $m$-summable Fredholm module. Furthermore, the
diagram \eqref{eq:diansum} of algebra homomorphisms is commutative. Now, let
$[x] \in K_m(A)$. By definition of the multiplicative character associated
with $\C F_1$ we have that
\[
\C M_{\C F_1}([x]) 
= \C M_{m-1}(  (\pi_{\C F_1})_*([x]))
= \big[ \C A_{m-1}([y]) \big] \in \cc/(2\pi i)^{\lceil m/2 \rceil} \zz
\]
Here $[y] \in K_m^{\T{rel}}(\C M^{m-1})$ is \emph{any} element such that
\[
\theta([y]) = (\pi_{\C F_1})_*([x]) \in K_m(\C M^{m-1})
\]
By naturality of the homomorphism $\theta : K_m^{\T{rel}}(\C M^{m-1}) \to
K_m(\C M^{m-1})$ and the commutativity of the diagram \eqref{eq:diansum} we
then get that
\[
(\theta \circ (\pi_{\C T})_* )([y]) = (\pi_{\C F_2})_*([x])
\]
We can thus deduce that
\[
\C M_{\C F_2}([x])
= \big[ \C A_{m-1}( (\pi_{\C T})_*([y])) \big] 
= \big[ \C A_{\C T}([y]) \big] \in \cc/(2 \pi i)^{\lceil m/2 \rceil} \zz
\]
But $\C T$ was an $m$-summable perturbation of the universal $m$-summable
Fredholm module. The desired result therefore follows from Theorem
\ref{pairaddi}.
\end{proof}

\end{document}